\documentclass[12pt]{amsart}
\usepackage{amsmath,amssymb,amsfonts,mathrsfs}
\usepackage{datetime}

 \newcommand{\obra}[3]{{\sc #1} {\em #2}. {#3}.}

\newtheorem{theorem-main}{Theorem}
\newtheorem{theorem}{Theorem}
 \newtheorem{lemma}[theorem]{Lemma}
 \newtheorem{proposition}[theorem]{Proposition}
 \newtheorem{fact}[theorem]{Fact}
 \newtheorem{corollary}[theorem]{Corollary}
 
 \theoremstyle{definition}
 \newtheorem{definition}[theorem]{Definition}
 \newtheorem{remark}[theorem]{Remark}
 \newtheorem{remarks}[theorem]{Remarks}
 \newtheorem{example}[theorem]{Example}

\theoremstyle{remark}
\newtheorem*{rmk}{Remark}

 \newcommand{\R}{\mathbb{R}}
 
 \newcommand{\N}{\mathbb{N}}
 
  \newcommand{\C}{\mathbb{C}}
 
 \renewcommand{\S}{\mathbb{S}}

 \newcommand{\FF}{\mathcal{F}}
 
  \newcommand{\OO}{\mathcal{O}}
 \newcommand{\CC}{\mathcal{C}}

   \renewcommand{\SS}{\mathcal{S}}
   \newcommand{\DD}{\mathcal{D}}
    
        \newcommand{\II}{\mathcal{I}}

\newcommand{\VV}{\mathcal{V}}
\newcommand{\BB}{\mathcal{B}}
 \renewcommand{\AA}{\mathcal{A}}

\newcommand{\trg}{|\g|}
\newcommand{\tr}[1]{\left|#1\right|}

\newcommand{\g}{\gamma}

\DeclareMathOperator{\an}{an}
\DeclareMathOperator{\ord}{ord}

\DeclareMathOperator{\gr}{gr}
\DeclareMathOperator{\trace}{trace}
\DeclareMathOperator{\diag}{diag}
\newcommand{\into}{\longrightarrow}

\newcommand{\rest}[1]{\!\!\upharpoonright_{#1}}

\newcommand{\Ps}[2]{\mathbb{#1}[\![#2]\!]}

\author{Olivier Le Gal}
\address{Universit\'e de Savoie \\ Laboratoire de Math\'ematiques \\ B\^atiment Chablais \\ Campus Scientifique \\ 73376 Le Bourget-du-Lac Cedex \\ France}
\email{Olivier.Le-Gal@univ-savoie.fr}

\author{Fernando Sanz}
\address{Universidad de Valladolid \\ Departamento de \'Algebra, An\'alisis Matem\'atico, Geometr\'ia y Topolog\'ia \\ Facultad de Ciencias \\ Campus Miguel Delibes \\ Paseo de Bel\'en, 7 \\ E-47011 Valladolid \\ Spain}
\email{fsanz@agt.uva.es}

\author{Patrick Speissegger}
\address{Department of Mathematics and Statistics \\
McMaster University \\
1280 Main Street West \\
Hamilton, Ontario L8S 4K1 \\
Canada}
\email{speisseg@math.mcmaster.ca}


\date{\today}


\title[O-minimality of interlaced trajectories]{Trajectories in interlaced integral pencils of 3-dimensional
analytic vector fields are o-minimal}

\begin{document}
	
	\begin{abstract}
		Let $\xi$ be an analytic vector field at $(\R^3,0)$ and $\II$ be an analytically
		non-oscillatory integral pencil of $\xi$; i.e., $\II$ is a maximal family of analytically non-oscillatory trajectories of $\xi$ at 0 all sharing the same iterated tangents. We prove that if
		$\II$ is interlaced, then for any trajectory $\Gamma\in\II$, the expansion
		$\R_{\an,\Gamma}$ of the structure $\R_{\an}$
		by $\Gamma$ is model-complete, o-minimal and polynomially bounded.
	\end{abstract}
	
	\maketitle
	
	\section{Introduction}
	We fix a real analytic vector field $\xi$ in a neighborhood
	$U$ of the origin $0 \in \R^n$, with $n \ge 2$, and suppose that $\xi(0)=0$. We are
	interested in solutions of $\xi$ of the form $\g:(a,b] \into U$ with $a \in [-\infty,b)$ and $b \in \R$; however, we are not
	interested in any particular parametrization of such a solution $\gamma$ but
	only in its image $$\trg:=\{\g(t):\ a < t \le b\},$$ which we will
	call a \textbf{trajectory of $\xi$}.  If, in addition, $\lim_{t \to a} \gamma(t) = 0$, we call $\trg$ a \textbf{trajectory at 0 of $\xi$}.  As in our more elementary paper \cite{LeG-San-Spe}, we are interested in the following vague questions:
	\begin{itemize}
		\item [(a)] What is the relative behavior between distinct trajectories at 0 of $\xi$?
		\item [(b)] What finiteness properties, relative to a given family of sets, do trajectories at 0 of $\xi$ have?
	\end{itemize}
	To make these questions precise in the cases considered here and to state our theorem, we need to recall, in the next two paragraphs, some terminology and results from Cano, Moussu and Sanz \cite{Can-Mou-San1, Can-Mou-San2}.  We assume the reader to be familiar with \textit{semianalytic} and \textit{subanalytic} sets (see for instance Bierstone and Milman \cite{Bierstone:1988ph}).
	
	Let $\g:(a,b] \into \R^n$ be a differentiable curve; for $c \in (a,b]$, we set $$\trg_c:= \{\g(t):\ a < t \le c\}.$$  We call $\g$ and its image $\trg$  \textbf{analytically non-oscillatory} if,
	for every semianalytic $A \subseteq \R^n$, there exists $c \in (a,b]$ such that either $\trg_c \subseteq A$
	or $\trg_c \cap A = \emptyset$.  Thus, one way to make question (b) precise is to ask, as done in \cite{Can-Mou-San1}, whether a given trajectory at 0 of $\xi$ is analytically non-oscillatory (simply called ``non oscillante'' there).  In \cite{Can-Mou-San1}, the notion of analytical non-oscillation is compared to the following: assume $\gamma(t) \ne 0$ for $t \in (a,b]$ and $\lim_{t \to a} \gamma(t) = 0$, and let 
	$\g_1:=\pi_1^{-1}\circ\g$ be the lifting of $\g$ via the blowing-up $\pi_1:M_1\to\R^n$ with center the
	origin $p_0=0$.  If $\gamma_1$ has a single limit point
	$p_1\in\pi_1^{-1}(p_0)$ as $t\to\infty$, we say that $\gamma$ \textbf{has tangent $p_1$ at} the origin.  We say that \textbf{$\g$ has iterated tangents at} the origin if, for $k \in \N$, there are differentiable curves $\gamma_k:(a,b] \into M_k$ and points $p_k \in M_k$ such that $M_0 = \R^n$, $\gamma_0 = \gamma$, $p_0 = 0$ and, for $k > 0$, $\gamma_k$ is the lifting of $\gamma_{k-1}$ via the blowing-up $\pi_k:M_k \into M_{k-1}$ with center $\{p_{k-1}\}$ and $\lim_{t \to a} \gamma(t) = p_k$.  In this situation, the \textbf{sequence of iterated tangents} $(p_k)_{k \in \N}$ thus obtained is
	uniquely determined by the image $\trg$.  By \cite[Section 1.2]{Can-Mou-San1}, if $\trg$ is analytically non-oscillatory, then $\gamma$ has iterated tangents; the converse is false in general, even if $n = 3$ and $\trg$ is a trajectory of $\xi$ at the origin \cite[Th\'eor\`eme 1]{Can-Mou-San1}. 
	
	
	The notions of the previous paragraph make sense for any $n \ge 2$.  To make sense of question (a) in the case $n=3$, we recall the following definitions from \cite{Can-Mou-San2}: let $\gamma, \gamma':(a,b] \into \R^3$ be  
	two analytically non-oscillatory, differentiable curves such that $\trg \cap \tr{\g'} = \emptyset$.  We say that they are \textbf{interlaced} if, for some system $(x,y,z)$ of analytic coordinates at the
	origin, there are $c,c' \ge a$, $\varepsilon > 0$ and differentiable functions $u,v,u',v':(0,\varepsilon] \into \R$ such that 
	$\trg_c=\{(x,u(x),v(x)):\ 0 < x \le \varepsilon\}$ and $\tr{\g'}_{c'}=\{(x,u'(x),v'(x)):\ 0 < x \le \varepsilon\}$, and such that the vector $(u(x)-u'(x),v(x)-v'(x))\in\R^2$ spirals
	around the origin as $x\to 0$. We say that $\trg,\tr{\g'}$ are
	\textbf{subanalytically separated} if there exists a subanalytic map
	$\sigma$ from a neighborhood of $\trg\cup\tr{\g'}$ into $\R^2$ such
	that $\sigma(\trg)\cap\sigma(\tr{\g'})$ is a finite set of points.
	
	The main result of \cite{Can-Mou-San2} relates these two notions in the following situation: an \textbf{integral pencil at 0 of $\xi$}
	is a maximal collection of trajectories at 0 of $\xi$ all having the same sequence of iterated tangents.  We call an integral pencil $\II$ at 0 of $\xi$ \textbf{analytically non-oscillatory} if every trajectory of $\II$ is analytically non-oscillatory.  In \cite[Th\'eor\`eme 1]{Can-Mou-San2} it is proved that, if $\II$ is an analytically non-oscillatory integral
	pencil at 0 of $\xi$, then either every pair of disjoint trajectories in $\II$ is interlaced, in which case we call $\II$ an \textbf{interlaced
		pencil}, or every pair of disjoint trajectories in $\II$ is subanalytically separated, in which case we call $\II$ a \textbf{subanalytically separated
		pencil}.
	
	For our theorem, we assume the reader to be familiar with the basics of \textit{o-minimal structures} (see van den Dries and Miller \cite{Dries:1996kf}); in particular, we will be working with the o-minimal structure $\R_{\an}$, whose definable sets are the \textit{globally subanalytic} sets.  For a trajectory $\Gamma$ at 0 of $\xi$, we let $\R_{\an,\Gamma}$ be the expansion of $\R_{\an}$ by $\Gamma$.  Clearly, the \hbox{o-minimality} of $\R_{\an,\Gamma}$ implies that $\Gamma$ is analytically non-oscillatory.  The converse is not true in general: while Rolin, Sanz and Sch\"afke \cite{Rol-San-Sch} give, in any dimension $n$, criteria for (and specific examples of) $\xi$ and analytically non-oscillatory trajectories $\Gamma$ at 0 of $\xi$ that imply the o-minimality of $\R_{\an,\Gamma}$, they also exhibit a particular $\xi$ in $\R^5$ with an analytically non-oscillatory trajectory $\Gamma$ at 0 such that $\R_{\an,\Gamma}$ is not o-minimal.  The question of whether counterexamples of the latter kind exist in $\R^3$ or $\R^4$ remains open, and our main theorem can be viewed as a partial result towards showing that no such counterexamples exist in $\R^3$:
	\medskip
	
	\noindent\textbf{Main Theorem.} 
	\textit{Let $\,\II$ be an interlaced, analytically non-oscillatory integral pencil at 0 of an analytic vector field $\xi$ on $\R^3$, and let
		$\,\Gamma$ be a trajectory of $\,\II$. Then the expansion
		$\mathbb{R}_{\an,\Gamma}$ of $\,\R_{\an}$ by $\Gamma$ is model complete, o-minimal and polynomially bounded.}
	\medskip
	
	Let $\II$ be an analytically non-oscillatory integral pencil at 0 of $\xi$.  An even stronger criterion than o-minimality of $\R_{\an,\Gamma}$, for individual trajectories $\Gamma \in \II$, is that of o-minimality of the expansion $\R_{\an,\II}$ of $\R_{\an}$ by \textit{all} trajectories in  $\II$.  For instance, if $n=2$, then $\R_{\an,\II}$ is o-minimal, because non-oscillatory trajectories at 0 of $\xi$ are
	pfaffian sets in this case, see Lion and Rolin \cite{Lio-Rol} or Speissegger \cite[Example 1.3]{Spe}. If $n=3$, however, the o-minimality of $\R_{\an,\II}$  and \cite[Th\'eor\`eme 1]{Can-Mou-San2} imply that $\II$ is subanalytically separated since, by its very
	definition, two interlaced trajectories cannot be definable in the
	same o-minimal structure. Thus, the Main Theorem above is the best we can hope for if $\II$ is interlaced.
	
	If $\II$ is subanalytically separated, we do not know what happens in general.  For the record, in \cite{LeG-San-Spe} we consider this problem in the case where $\xi$ arises from a system of two linear ODEs with meromorphic  coefficients
	\begin{equation*}
	y'=A(x)y + B(x), \quad y=(y_1,y_2).
	\end{equation*}
	In this situation, we obtain from \cite[Theorem 4]{LeG-San-Spe} that if $\II$ is a subanalytically separated integral pencil at $0$ of $\xi$, then the expansion of $\R_{\an}$ by all trajectories in $\II$ is o-minimal.
	
	The proof of the Main Theorem goes as follows:
	in Section \ref{reduce}, we use a result in \cite{Can-Mou-San2} to reduce to the situation where the vector field $\xi$ arises from a two-dimensional system of differential equations in \textit{final form}.  Basic ODE theory then gives the existence of a formal power series solution $H(X)$ of this system to which the trajectories $\Gamma$ we are interested in are asymptotic.  In this situation, a result of Rolin, Sanz and Sch\"afke \cite{Rol-San-Sch} states that $\R_{\an,\Gamma}$ is o-minimal provided $H(X)$ satisfies the so-called SAT property (see Section \ref{rss}).  Thus, similar to \cite{Rol-San-Sch}, it remains to establish this SAT property of $H(X)$.  In \cite{Rol-San-Sch}, this was achieved under the additional assumption that $\xi$ has sufficiently many \textit{independent} (over the non-flat germs) components of Stokes phenomena (see Sections \ref{multi} and \ref{multisummability} for definitions).  The main contribution of this paper is the independence proof of the components of the Stokes phenomena in the situation considered here, from which we then obtain the SAT property along the lines of \cite{Rol-San-Sch}, carried out in Section \ref{sat}.  This independence proof, in turn, is based on a further reduction to what we call ``interlaced final form'' (Proposition \ref{int-final-form-reduction}), as well as on multisummability theory, see Example \ref{explicit_stokes} and Proposition \ref{pro:multisum}.

	\section{Reduction to interlaced final form} \label{reduce}
	
	\subsection*{Systems of ODEs}
	To describe the first reduction in the proof of our Main Theorem, we work in the following setting: we fix $q \in \N$ and nonzero $n \in \N$ and consider an $n$-dimensional system of ordinary differential equations of the form
	\begin{equation}
	\label{eq:system-gen}
	x^{q+1} y'(x) = \Theta(x,y(x)),
	\end{equation}
	where $y \in \R^n$ and $\Theta:V \into \R^n$ is real analytic in some neighbourhood $V$ of $0 \in \R^{1+n}$.  A \textbf{solution at 0 of \eqref{eq:system-gen}} is a differentiable map $y:(0,\epsilon] \into \R^n$, for some $\epsilon > 0$, such that $\gr y \subseteq V$ and $y$ satisfies \eqref{eq:system-gen} for $0 < x \le \epsilon$.  A \textbf{formal solution at 0 of \eqref{eq:system-gen}} is an $n$-tuple $H \in \Ps{R}{X}^n$ such that $(0,H(0)) \in V$ and $$X^{q+1} H'(X) = (T_{(0,H(0))}\Theta)(X,H(X)-H(0)),$$ where $T_a\Theta \in \Ps{R}{X,Y}$ denotes the Taylor series of $\Theta$ at $a \in V$ and $Y = (Y_1, \dots, Y_n)$.
	
	\begin{rmk}
		The integer $q$ is equal to the Poincar\'{e} rank of system \eqref{eq:system-gen} if $T_{(0,H(0))}\Theta$ is not divisible by $X$ in $\Ps{R}{X,Y}$.
	\end{rmk}
	
	Let $\eta = -x^{q+1 }\partial_x - \Theta(x,y) \cdot \partial_y$ be the real analytic vector field, defined in a neighbourhood of $0 \in \R^{1+n}$, associated to  system \eqref{eq:system-gen}, where $\partial_y = (\partial_{y_1}, \dots, \partial_{y_n})$.  Then the graph of any solution $h$ at $0$ of system \eqref{eq:system-gen} is a trajectory $\Gamma$ of $\eta$.
	
	\begin{rmk}
		This $\Gamma$ is not necessarily a trajectory at 0 of $\eta$; indeed, the graph of $h$ above is a trajectory at 0 of $\eta$ if and only if $\lim_{t \to 0^+} h(t)=0$. 
	\end{rmk}
	
	Thus, we call a solution $h$ at $0$ of system \eqref{eq:system-gen} \textbf{analytically non-oscillatory} if its graph $\gr h$ is analytically non-oscillatory.  In addition, if $n=2$, we call a pair $(g,h)$ of solutions at $0$ of system \eqref{eq:system-gen} \textbf{subanalytically separated} (respectively, \textbf{interlaced}) if the pair of graphs $(\gr g, \gr h)$ is subanalytically separated (respectively, interlaced).  
	
	\begin{remark}
		\label{it_tang_rmk}
		Assume that system \eqref{eq:system-gen} has a formal solution $H$ at $0$.  We set $H^0:= H$, $p_0 := H(0) \in \R^n$ and $H^1(X):= (H(X)-p_0)/X \in \Ps{R}{X}^n$.  Iterating this procedure we obtain, by induction on $k \in \N$, points $p_k \in \R^n$ and tuples $H^{k} \in \Ps{R}{X}^n$ such that $p_{k} = H^k(0)$ and $H^{k+1}(X) = \left(H^{k}(X) - p_{k}\right) / X$.  If $h$ is a solution at $0$ of system \eqref{eq:system-gen} with asymptotic expansion $H$ at 0, then this computation corresponds to the computation of the iterated tangents of the graph of $h$ in suitable charts at each stage of blowing up.  Therefore, a solution $h$ at $0$ of system \eqref{eq:system-gen} has asymptotic expansion $H$ at 0 if and only if the graph of $h$ has iterated tangents at $0$ determined by $H$ through the above computation (see the recent survey \cite{San} for details).
	\end{remark}
	
	Thus, we call \textbf{integral pencil at $0$ of system \eqref{eq:system-gen}} any maximal collection of solutions at 0 of system \eqref{eq:system-gen} all having the same asymptotic expansion at 0.  In particular, if the system \eqref{eq:system-gen} has a formal solution $H$ at $0$, we denote by $\II(H)$ the integral pencil of system \eqref{eq:system-gen} consisting of all solutions at 0 of \eqref{eq:system-gen} asymptotic to $H$.  
	
	In addition, if $n=2$, we call an integral pencil $\II$ at $0$ of system \eqref{eq:system-gen} \textbf{analytically non-oscillatory} if every solution in $\II$ is analytically non-oscillatory, and we call $\II$ \textbf{subanalytically separated} (respectively, \textbf{interlaced}) if every pair of distinct solutions in $\II$ is subanalytically separated (respectively, interlaced).  
	
	\begin{rmk}
		It follows from Remark \ref{it_tang_rmk} that, if $h$ is a solution at $0$ of system \eqref{eq:system-gen} with asymptotic expansion $H$ and $\II$ is the integral pencil containing $h$, then $\II = \II(H)$.  
	\end{rmk}
	
	\subsection*{The reduction}
	We assume for the remainder of this section that $n=2$.  Following \cite[D\'efinition 4.2]{Can-Mou-San1}, we say that system \eqref{eq:system-gen} is in \textbf{final form} if $q \ge 1$ and
	\begin{equation}\label{eq:system-normal-form}
	\Theta(x,y)=\big(a(x)I+x^rM(x)\big)y + x^{q+1}g(x,y),
	\end{equation}
	where $0 \le r\leq q+1$ (this $r$ corresponds to the ``indice de radialit\'e'' $k(X)$ in \cite[D\'efinition 4.2]{Can-Mou-San2}), $g$ is real analytic in some neighbourhood of $0$, $a(x)$ is a polynomial of degree at most $r-1$ (with $a(x) = 0$ if $r = 0$), $I$ is the identity matrix, and $M(x)$ is a matrix of polynomials of degree at most $q-r$ (with $M(x) = 0$ if $r > q$), such that the matrix $A(x):= a(x)I + x^r M(x)$ has at least one nonzero eigenvalue at $x=0$ and $M(0)$ has two distinct eigenvalues if $r \le q$.  With these notations, the case $r=q+1$ corresponds to 
	$$\Theta(x,y) = a(x) Iy + x^{q+1}g(x,y),$$
	that is, to a system whose ``indice de radialit\'e'' is bigger than the Poincar\'e rank $q$.

	Assume that system \eqref{eq:system-gen} is in final form \eqref{eq:system-normal-form}.  The hypothesis on the eigenvalues of $A(x)$ at $x=0$ then imply (by a routine calculation as found, for instance, in Chow and Hale \cite[Chapter 12, Theorem 3.7]{Cho-Hal}) that there exists a unique formal solution $H$ at 0 of system \eqref{eq:system-gen} such that $H(0) = 0$.  Moreover, by Bonckaert and Dumortier \cite[Theorem 2.1]{Bon-Dum}, there exists a solution $h$ at 0 of system \eqref{eq:system-gen} with asymptotic expansion $H$ at 0; in particular, $\II(H)$ is nonempty.

	\begin{fact}[\cite{Can-Mou-San2}, Th\'eor\`emes 4.3 and 4.5]
		\label{pro:separated-or-interlaced} 
		Assume that system \eqref{eq:system-gen} is in final form \eqref{eq:system-normal-form}, and let $H$ be its unique formal solution at $0$ satisfying $H(0) = 0$.  Then $\II(H)$ is analytically non-oscillatory and interlaced if and
		only if $H$ is divergent and the following holds:
		\begin{align}
		\label{eq:can-mou-san} 
		\begin{split}
		&M(0) \text{ has non-real eigenvalues and }\\ &\trace A(x)=\alpha x^l + O(x^{l+1}) \text{ for some } l<q \text{ and } \alpha>0.
		\end{split}
		\end{align}
		Moreover, in this situation, the integral pencil $\II(H)$ consists of all solutions $h$ at 0 satisfying $\,\lim_{x\to 0^+}h(x)=0$.
	\end{fact}
	
	
	To see how this fact is used towards the proof of our Main Theorem, let $\II$ be an interlaced, analytically non-oscillatory integral pencil at 0 of a given analytic vector field $\xi$.  By \cite[Proposition 5.1]{Can-Mou-San1}, there exists a polynomial map
	$\sigma:\R^3 \into \R^3$ fixing the origin (obtained by a finite composition of
	local blowings-up and ramifications), and there exists a system \eqref{eq:system-gen} in final 
	form, with unique formal solution $H$ at $0$ satisfying $H(0) = 0$, such that every
	trajectory $\trg$ in $\II$ is the image under $\sigma$ of the graph of some solution in the pencil $\II(H)$ of this system \eqref{eq:system-gen}. Since the map $\sigma$ is
	polynomial, it follows from \cite[Proposition 1.13]{Can-Mou-San1} that $\II(H)$ is non-oscillatory and interlaced.  
	
	Moreover, if $H$ is a formal solution at 0 of a system \eqref{eq:system-gen} in final form satisfying \eqref{eq:can-mou-san} and $H(0) = 0$, a routine linear change of variables $y \mapsto Ry$, where $R \in \mathcal{M}_2(\R)$, shows that $R H$ is a formal solution at 0 of a system \eqref{eq:system-gen} in final form \eqref{eq:system-normal-form} satisfying \eqref{eq:can-mou-san} and the following additional condition:
	\begin{equation}\label{eq:Jx2}
	M(0)=\begin{pmatrix} \mathfrak{a} & -\mathfrak{b} \\ \mathfrak{b} & \mathfrak{a}\end{pmatrix}, \text{ where } \mathfrak{a}, \mathfrak{b} \text{ are real and } \mathfrak{b} \neq 0.
	\end{equation}
	
	These observations lead us to the following definition: we say that system \eqref{eq:system-gen} is in \textbf{interlaced final form} if $q \ge 1$, $r \le q$ and
	\begin{equation}
	\label{eq:int-final-form}
	\Theta(x,y) = \left(a(x)I + x^r b(x)J\right) y + x^{q+1} g(x,y) + c(x),
	\end{equation}
	where $a(x) = a_0 + \cdots + a_q x^q$ is a polynomial of degree at most $q$ satisfying $a_l > 0$ for the least $l$ such that $a_l \ne 0$, $b(x) = b_0 + \cdots + b_{q-r} x^{q-r}$ is a polynomial of degree at most $q-r$ satisfying $b_0 \ne 0$, $c(x)$ is a tuple of polynomials of degree at most $q$ satisfying $c(0) = 0$, $g$ is real analytic in some neighbourhood of $0$ and $J = \begin{pmatrix} 0 & -1 \\ 1 & 0  \end{pmatrix}$.  
	
	\begin{remark}
		\label{int-final-form-rmk}
		The explanation for the additional term $c(x)$ is deferred to Remark \ref{0-sat_reduction}.  A system \eqref{eq:system-gen} in interlaced final form \eqref{eq:int-final-form} with $c(x) = 0$ is in final form \eqref{eq:system-normal-form} and satisfies conditions \eqref{eq:can-mou-san} and \eqref{eq:Jx2}.  Moreover, the arguments given before Fact \ref{pro:separated-or-interlaced} also apply to any system \eqref{eq:system-gen} in interlaced final form, i.e., for any such system, there exists a unique formal solution $H$ at 0 such that $H(0) = 0$, and there exists a solution $h$ at 0 with asymptotic expansion $H$ at 0, so that $\II(H)$ is nonempty.
	\end{remark}
	
	\begin{proposition}
		\label{int-final-form-reduction}
		Assume that system \eqref{eq:system-gen} is in final form \eqref{eq:system-normal-form} and satisfies conditions \eqref{eq:can-mou-san} and \eqref{eq:Jx2}.  Then there exist $T_1, \dots, T_q \in \mathcal{M}_2(\R)$ such that, with $$T(x):= I + x T_1 + \cdots + x^q T_q,$$ the pullback of system \eqref{eq:system-gen} via the change of variables $y = Tz$ for $z \in \R^2$ is in interlaced final form with corresponding $c(x) = 0$.
	\end{proposition}
	
	\begin{proof}
		Set again $A(x):= a(x)I + x^rM(x)$, and assume $A$ satisfies conditions \eqref{eq:can-mou-san} and \eqref{eq:Jx2}.  If a matrix $T$ as required exists, then there exists a real analytic $g_T$, defined on a neighbourhood of $0$ and depending on $T$, such that $h$ is a solution at 0 of our system \eqref{eq:system-gen} if and only if $T^{-1} h$ is a solution at 0 of the system
		\begin{equation*}
		x^{q+1} z' = T^{-1} \left( AT - x^{q+1} T'\right) z + x^{q+1} g_T(x,z).
		\end{equation*}
		Thus, it suffices to find $T$ and matrices $D,E \in \mathcal{M}_2(\R)[x]$ of degree at most $q$ such that
		\begin{equation}
		\label{eq:T-eq}
		AT - TD - x^{q+1} T' = x^{q+1} E
		\end{equation}
		and
		\begin{equation}
		\label{eq:D-form}
		D(x) = a(x)I + x^rN(x),
		\end{equation}
		where $N(x) = N_0 + xN_1 + \cdots + x^{q-r} N_{q-r}$, with each $N_j \in \mathcal{M}_2(\R)$ of the form $\begin{pmatrix} \mathfrak{a}_j & -\mathfrak{b}_j \\ \mathfrak{b}_j & \mathfrak{a}_j \end{pmatrix}$ and $\mathfrak{b}_0 \ne 0$.  To do so, we write $M(x) = M_0 + x M_1 + \cdots + x^{q-r}M_{q-r}$ with each $M_j \in \mathcal{M}_2(\R)$.  Plugging into \eqref{eq:T-eq} yields
		\begin{align*}
		&x^{q+1} E & = &x^r (MT - TN) - x^{q+1} T' \\
		&& = &x^r (M_0 - N_0) \\ &&&+ x^{r+1} (M_0 T_1 - T_1 N_0 + M_1 - N_1) \\ &&&\ \vdots \\ &&&+ x^q\left( \sum_{j=0}^{q-r-1} (M_j T_{q-r-j} - T_{q-r-j} N_j) + M_{q-r} - N_{q-r} \right) \\ &&&+ x^{q+1} P,
		\end{align*}
		where $P \in \mathcal{M}_2(\R)[x]$ is of degree at most $q$ and depends on $T$ and $N$.  This shows that we can take $N_0:= J$, which works because of our hypotheses.  Working by induction on $k = 0, \dots, q-r$, we therefore assume $k > 0$ and having found $T_1, \dots, T_{k-1}$ and $N_0, \dots, N_{k-1}$ with the required properties such that $$\sum_{j=0}^{l-1} (M_j T_{l-j} - T_{l-j} N_j) + M_{l} - N_{l} = 0, \quad\text{for } l=0, \dots, k-1;$$ we then need to find $T_k$ such that $$N_k := M_k + (M_0 T_{k} - T_{k} N_0) + \sum_{j=1}^{k-1} (M_j T_{k-j} - T_{k-j} N_j)$$ also has the required properties.  Since the matrix
		\begin{equation*}
		\begin{pmatrix} \alpha & \beta \\ \gamma & \delta \end{pmatrix} := M_k + \sum_{j=1}^{k-1} (M_j T_{q-r-j} - T_{q-r-j} N_j)
		\end{equation*}
		is already determined, direct computation shows that $$T_k:= \frac1{4 \mathfrak{b}} \begin{pmatrix} -\gamma-\beta & \alpha - \delta \\ \alpha - \delta & \gamma+\beta \end{pmatrix}$$ does the job.  Finally, with $T$ and $N$ determined in this way, both $P$ and $g_T$ are determined as well, and we take $E:= P$.  Then  $$D = T^{-1}AT + O\left(x^{q+1}\right),$$ so that tr$(D)=$ tr$(A)+O\left(x^{q+1}\right)$. Since $A$ satisfies \eqref{eq:can-mou-san}, it follows that the condition $a_l>0$ in the definition of interlaced final form \eqref{eq:int-final-form} is met.
	\end{proof}
	
	Thus, the Main Theorem is implied by the following particular case:
	\begin{theorem}\label{th:main-final-form}
		Assume that system \eqref{eq:system-gen} is in interlaced final form \eqref{eq:int-final-form}, and let $H$ be its unique divergent formal solution at 0 satisfying $H(0) = 0$. Then, for $h \in \II(H)$, the structure $\R_{\an,h}$ is model complete, \hbox{o-minimal} and
		polynomially bounded.
	\end{theorem}

	\section{Reduction to establishing SAT} \label{rss}
	
	To explain our variation of the approach in \cite{Rol-San-Sch}, we need to recall some definitions and facts.  First, recall that a tuple $F = (F_1, \dots, F_l) \in \Ps{R}{X}^l$ such that $F(0) = 0$ is \textbf{analytically transcendental} if, for every \textit{convergent} $G \in \Ps{R}{X,Z}$ such that $G(0) = 0$ and $Z = (Z_1, \dots, Z_l)$, the condition $G(X, F(X))=0$ implies $G=0$.
	
	For the remainder of this section, we work with system \eqref{eq:system-gen} and assume that it has a formal solution $H$ at $0$.  For $k \in \N$, we associate the point $p_k \in \R^n$ and the tuple $H^k$ to $H$ as in Remark \ref{it_tang_rmk}, and we set $$R_k H(X) = (R_k H_1(X), \dots, R_k H_n(X)):= H^k(X)-p_k.$$  Note that $R_k H(0) = 0$ for each $k$.
	
	\begin{definition}
		\label{def:SAT}
		Let $q$ be as in system \eqref{eq:system-gen}.
		\begin{enumerate}
			\item We call a polynomial $P \in\R[X]$ \textbf{positive} if
			$P(x)>0$ for all sufficiently small $x>0$, and we call $P$ \textbf{$q$-short} if $\deg P < (q+1) \ord P$.  
			\item The formal solution $H$ is \textbf{strongly analytically transcendental}, or \textbf{SAT} for short (pronounced ``sat''), if for any integers $k \ge 0$ and $l \ge 1$ and
			any $l$-tuple $P = (P_1, \dots, P_l)$ of distinct $q$-short positive polynomials, the tuple $$R_k H \circ P:= \left(R_k H_1 \circ P_1, \dots, R_k H_n \circ P_1, R_k H_1 \circ P_2, \dots, R_k H_n \circ P_k\right)$$ is analytically transcendental. 
		\end{enumerate}
	\end{definition}
	
	\begin{fact}[Lemma 4.1 and Theorem 2.2 of \cite{Rol-San-Sch}]
		Assume that system
		\eqref{eq:system-gen} has a SAT formal solution $H$ at 0. Then for every
		$h \in \II(H)$, the structure
		$\R_{\an,h}$ is model-complete, o-minimal and
		polynomially bounded.
	\end{fact}
	
	Thus, to prove Theorem \ref{th:main-final-form} (and hence the Main Theorem), it suffices to establish the following:
	
	\begin{theorem}\label{th:main-SAT}
		Assume that system \eqref{eq:system-gen} is in interlaced final form \eqref{eq:int-final-form}, and let $H$ be its unique divergent formal solution at 0 satisfying $H(0) = 0$. Then $H$ is SAT.
	\end{theorem}
	
	Let us point out that, in the situation of Theorem \ref{th:main-SAT} with $r=0$ in \eqref{eq:int-final-form}, system \eqref{eq:system-gen} also satisfies
	the hypotheses in \cite[Theorem 2.4']{Rol-San-Sch}, thus implying Theorem \ref{th:main-SAT} for this case.  In general, however, we allow the linear part of \eqref{eq:int-final-form} to have two real eigenvalues (whenever $r>0$), a case to which
	\cite[Theorem 2.4']{Rol-San-Sch} does not apply.  As our proof would not be different for the case $r=0$, we shall focus on the case $r>0$, which allows us to somewhat lighten notations.  
	
	The reason for the term $c(x)$ in the definition of ``interlaced final form'' is that it suffices to establish the following weakening of SAT:
	
	\begin{definition}
		\label{0-sat}
		Let $q$ be as in system (1).  The formal solution $H$ is \textbf{0-SAT} if for any integer $l \ge 1$ and
		any $l$-tuple $P = (P_1, \dots, P_l)$ of distinct $q$-short positive polynomials, the tuple $R_0 H \circ P$ is analytically transcendental.
	\end{definition}
	
	\begin{remark}
		\label{0-sat_reduction}
		It suffices to prove Theorem \ref{th:main-SAT} with ``0-SAT'' in place of ``SAT''.  To see this, assume that Theorem \ref{th:main-SAT} holds with ``0-SAT'' in place of ``SAT'', and assume that system \eqref{eq:system-gen} is in interlaced final form \eqref{eq:int-final-form}, and let $H$ be its unique formal solution at 0 satisfying $H(0) = 0$.  Then 
		\begin{equation}
		\label{rmk 8}
		X^{q+1} H' = T_0A \cdot H + X^{q+1} \cdot T_0 g(X,H) + T_0c,
		\end{equation}
		where $A(x):= a(x)I + x^r b(x)J$.  Since $R_1H(X) = H^1(X) - p = H(X)/X - p$, where $p:= H^1(0)$, it follows that
		\begin{align*}
		X^{q+1} (R_1&H)' = (T_0A-X^qI) H^1 + X^q \cdot T_0g(X,H) + T_0c/X \\
		&= (T_0A-X^qI) R_1H + X^{q+1} T_0h(X,R_1H) + T_0d,
		\end{align*}
		where $h(x,y):= \left(T_{(0,p)}G\right)(x,y)$ with $G(x,y):= (g(x,xy)-g(0,0))/x$ and $$d(x):= c(x)/x + x^q g(0,0) + (T_0A(x)-x^qI) p.$$  Note that $\deg d \le q$; dividing \eqref{rmk 8} by $X$ and setting $X=0$, we get that $d(0) = 0$.  Thus, $R_1H$ is the unique formal solution at 0, with $R_1H(0) = 0$, of another system \eqref{eq:system-gen} in interlaced final form \eqref{eq:int-final-form}.  Since $R_{k+1}H = R_1(R_kH)$ for $k \in \N$, we obtain, by iterating this procedure and  applying the hypothesis, that $H$ is SAT.
	\end{remark}

	\section{Summability}  \label{multi}
	
	We recall, in this and the next section, the basics of multisummability as described by Malgrange and Ramis \cite{Mal-Ram}, with notations adapted to our situation.  Thus, we set $\C^*:= \C \setminus \{0\}$, $\R^*:= \R \setminus \{0\}$, $\R_+:= [0,+\infty)$, $\R^*_+:= \R^* \cap \R_+$ and let $\S^1$ be the unit circle in $\R^2$.  We identify $\S^1$ with the interval $[0,2\pi)$ via the standard argument map, and we equip $\S^1$ in this way with addition $\oplus$ and subtraction $\ominus$ obtained from the corresponding operations modulo $2\pi$ on $[0,2\pi)$.  We also identify $\C^*$ with $\R^*_+\times \S^1$
	via the usual covering map
	$\rho:(r,\theta)\in\mathbb{R}_+\times\mathbb{S}^1\mapsto r
	e^{i\theta}$. 
	
	Thus, we associate to any subset $X$ of
	$\mathbb{S}^1$ the set $\mathcal{V}_X$ of all open $U \subseteq \C^*$ for which there exist an open $W \subseteq X$ and an $\epsilon > 0$ such that $\rho((0,\epsilon) \times W) \subseteq U$.
	For any $X \subseteq \S^1$, we let $\OO(X)$ be the algebra of all germs at 0 of analytic functions $f:U \into \C$ with $U \in \VV_X$; then $\OO:= \{\OO(U):\ U \subseteq \S^1$ open$\}$ is a sheaf on $\S^1$.  
	
	The reason for introducing sheaf terminology is that it provides a convenient setting in which to define multisummability; we refer the reader to Hartshorne \cite[Section II.1]{Hartshorne:1977hr} for details on sheaves.  Thus, we let $\AA$ be the subsheaf of $\OO$ whose stalk $\AA_\theta$, for $\theta \in \S^1$, consists of all $f \in \OO_\theta$ that have an asymptotic expansion $T_\theta f(X) = \sum a_n X^n \in \Ps{C}{X}$ at $0$, that is, there exist a representative $f:V \into \C$, with $V \in \VV_{\{\theta\}}$, and constants $c_n \in \R$ depending on $V$, for $n \in \N$, such that 
	\begin{equation}
	\label{asymptotic_estimate}
	\left| f(z) - \sum_{n=0}^{m-1} a_n z^n \right| \le c_m|z|^m, \quad\text{for } z \in V \text{ and } m \in \N.
	\end{equation}  
	If $\CC$ is the sheaf on $\S^1$ whose section, for open $U \subseteq \S^1$, consists of all locally constant maps $F:U \into \Ps{C}{X}$, we call \textbf{Taylor map} the morphism $T:\AA \into \CC$ of sheaves induced by the maps $T_\theta$.
	
	\begin{rmk}
		If $U \subseteq \S^1$ is connected, then $T\rest{\AA(U)}$ takes values in $\Ps{C}{X}$.  It follows from basic complex analysis that if $f \in \AA(\S^1)$, then $Tf$ converges.
	\end{rmk}
	
	Next, we define the subsheaf $\mathcal{A}^{0}$ of \textbf{flat functions} as the kernel of $T$ and, for $k>0$, we let $\AA^{k}$ be the subsheaf of $\AA^0$ whose stalk $\AA^k_\theta$, for $\theta \in \S^1$, consists of all $f \in \AA_\theta$ that are \textbf{exponentially flat of order at least $k$}, that is, there exist a representative $f:V \into \C$, with $V \in \VV_{\{\theta\}}$,  and constants $A,b>0$ depending on $V$ such that
	\begin{equation*}
	|f(z)|\le A e^{-b/|z|^k} \quad\text{for } z \in V.
	\end{equation*}
	
	\begin{fact}[\textbf{Watson's Lemma}, statement before D\'efinition 1.5 in \cite{Mal-Ram}]
		\label{watson}
		Let\footnote{If one replaces $\S^1$ by its universal covering space $\R$, all definitions and facts stated in this section are easily adapted to all $k>0$.  Since we only consider integer $k>0$ in this paper, the present setting suffices for our purposes.} $k > 1/2$ and $I \subseteq \S^1$ be a closed interval of length $|I| \ge \pi/k$.  Then $\AA^k(I) = \{0\}$.
	\end{fact}
	

	\subsection*{Gevrey asymptotics}
	Let $s \ge 0$.  We let $\Ps{C}{X}_s$ be the ring of all \textit{Gevrey series} of order $s$, that is, all $F(X) = \sum_{n=0}^\infty a_n X^n \in \Ps{C}{X}$ such that the series $\sum_{n=0}^\infty \frac{a_n}{\Gamma(ns)} X^n$ converges, where $\Gamma$ denotes the usual Gamma function.  We also let $\AA_s$ be the subsheaf of $\AA$ whose stalk $\AA_{s,\theta}$, for $\theta \in \S^1$, consists of all $f \in \AA_\theta$ for which there exist a representative $f:V \into \C$, with $V \in \VV_{\{\theta\}}$, and a constant $c>0$ depending on $V$ such that \eqref{asymptotic_estimate} holds with $c_n = c^n\Gamma(ns)$.
	Note that, for connected $U \subseteq \S^1$, we have $T(\AA_s(U)) \subseteq \Ps{C}{X}_s$.  
	
	\begin{fact}[1.3 and 1.4 of \cite{Mal-Ram}]
		\label{simple_summability}
		Let $k > 1/2$ and $I \subseteq \S^1$ be an interval.
		\begin{enumerate}
			\item $\AA_{1/k}(I) \cap \AA^0(I) = \AA^k(I)$.
			\item If $I$ is closed and of length less than $\pi/k$, then $T\rest{\AA_{1/k}(I)}$ is surjective onto $\Ps{C}{X}_{1/k}$.
			\item \textbf{Quasi-analyticity}: if $I$ is closed and of length at least $\pi/k$, then $T\rest{\AA_{1/k}(I)}$ is injective.
		\end{enumerate}
	\end{fact}
	
	One of the key concepts needed is that of \textit{quotient sheaf}.  In our situation, we have the following: if $\BB$ is a subsheaf of $\AA$ and $I$ is a subinterval of $\S^1$, then every element of $(\AA/\BB)(I)$ is \textbf{represented} by a (finite if $I$ is closed, possibly infinite if $I$ is not closed) tuple of elements $f_i \in \AA(U_i)$, such that each $U_i$ is an open interval, $I \subseteq \bigcup_i U_i$ and, for all $i,j$, we have $(f_i-f_j)\mid_{U_i \cap U_j} \in \BB(U_i \cap U_j)$.  
	
	Since $\AA^0$ is the kernel of $T$ and $\AA^k$ is a subsheaf of $\AA^0$, for $k\ge0$, the Taylor map induces a morphism $T_k:\AA/\AA^k \into \CC$ of sheaves; we usually omit the subscript $k$.  Moreover, we have
	
	\begin{corollary}
		\label{ramis-sibuya}
		The map $T:\left(\AA/\AA^k\right)\left(\S^1\right) \into \Ps{C}{X}_{1/k}$ is an isomorphism.
	\end{corollary}
	
	\begin{proof}
		By \cite[Th\'eor\`eme 1.6]{Mal-Ram}, we have $\left(\AA/\AA^k\right)\left(\S^1\right) = \left(\AA_{1/k}/\AA^k\right)\left(\S^1\right)$; the corollary then follows from Fact \ref{simple_summability}.
	\end{proof}
	
	\subsection*{Summability}
	To describe what we use from summability theory, we need the following notations: for distinct $\theta, \zeta \in \S^1$ and $k \ge 1$, we set $$d(\theta,\zeta):= \min\{\theta \ominus \zeta, \zeta \ominus \theta\} \in [0,\pi]$$ and $$V(\theta,k):= \left(\theta \ominus \frac{\pi}{2k},\theta \oplus \frac{\pi}{2k}\right);$$ so $V(\theta,k)$ is a proper subinterval of $\S^1$, and we denote its topological closure in $\S^1$ by $I(\theta,k)$.  If $d(\theta,\zeta) < \pi$, we let $U(\theta,\zeta)$ be the unique open interval in $\S^1$ with endpoints $\theta$ and $\eta$ and of length equal to $d(\theta,\eta)$.  If $d(\theta,\zeta) < \pi$, we set
	$$U(\theta, \zeta,k):= \bigcup_{\phi \in U(\theta,\zeta)} V(\phi,k);$$  note that, under these assumptions, $U(\theta,\zeta,k)$ is a proper subinterval of $\S^1$ of length greater than $\pi/k$.
	
	Let $k \ge 1$ and $F \in \Ps{C}{X}_{1/k}$.  Recall \cite[D\'efinition 1.5]{Mal-Ram} that, if $I \subseteq \S^1$ is a closed interval of length at least $\pi/k$,  then $F$ is \textbf{$k$-summable on $I$} if there exists $f\in\mathcal{A}_{1/k}(I)$ such that $Tf=F$.  By quasianalyticity, if such an $f$ exists, it
	is unique; we call it the \textbf{$k$-sum of $F$ on 
		$I$} and denote it by $\SS_I F$.

	\begin{definition}
		\label{stokes_def}
		\begin{enumerate}
			\item The series $F$ is \textbf{$k$-summable in the direction $\theta \in \S^1$} if $F$ is $k$-summable on $I(\theta,k)$.
			\item The series $F$ is \textbf{$k$-summable} if it is $k$-summable in all but finitely many directions; in this situation, the directions in which $F$ is not $k$-summable are called the \textbf{singular} directions of $F$.
			\item If $F$ is $k$-summable and $\xi,\zeta \in \S^1$ are such that $d(\xi,\zeta) < \pi$, and if the interval $U(\xi,\zeta)$ contains no singular directions of $F$ then, by analytic extension, there exists a unique $f \in \AA_{1/k}(U(\xi,\zeta,k))$ such that $f\rest{I(\theta,k)} = \SS_{I(\theta,k)} F$, for $\theta \in U(\xi,\zeta)$.  We call this $f$ the \textbf{$k$-sum of $F$ on $U(\xi,\zeta)$} and denote it by $\SS_{\xi,\zeta} F$. 
		\end{enumerate}\medskip
		
		Next, let $S \subseteq \S^1$ be finite; for $\theta \in \S^1$, we let $\theta^+(S)$ be the first element of $S \cup \{\theta \oplus \pi/2\}$, distinct from $\theta$, that lies on $\S^1$ after $\theta$ in the positive sense and, similarly, we let $\theta^-(S)$ be the first element of $S \cup \{\theta \ominus \pi/2\}$, distinct from $\theta$, that lies on $\S^1$ after $\theta$ in the negative sense.  Note that, for $\theta \in \S^1$ and $\ast \in \{+,-\}$, we have $d(\theta,\theta^\ast(S)) < \pi$ and 
		\begin{equation*}
		V(\theta,k) = U(\theta,\theta^-(S),k) \cap U(\theta,\theta^+(S),k),
		\end{equation*}
		independent of $S$.
		
		Assume now that $F$ is $k$-summable with its singular directions in $S$.  By definition, for $\theta \in \S^1$ and $\ast \in \{+,-\}$, the interval $U(\theta,\theta^\ast(S))$ contains no singular directions of $F$, so the $k$-sum $\SS_{\theta,\theta^\ast(S)}F$ is well defined.  The difference $$\Delta_\theta F := \SS_{\theta,\theta^+(S)} F - \SS_{\theta,\theta^-(S)} F$$ is defined on $V(\theta,k)$, independent of $S$ and called the \textbf{Stokes phenomenon of $F$ in the direction $\theta$}.  Note that $\Delta_\theta F = 0$ whenever $\theta \notin S$. 
		
		The tuples $\left(\SS_{\theta,\theta^+(S)} F\right)_{\theta \in \S^1}$ and $\left(\SS_{\theta,\theta^-(S)} F\right)_{\theta \in \S^1}$ are uniquely determined by $F$ and $S$.
		Moreover, by Fact \ref{simple_summability}(1), each $\Delta_\theta F$ belongs to $\AA^k(V(\theta,k))$.  It follows that the tuple $\left(\SS_{\theta,\theta^+(S)} F\right)_{\theta \in \S^1}$ represents an element in $\left(\AA_{1/k}/\AA^k\right)\left(\S^1\right)$, which we denote by $\SS F$ and call the \textbf{$k$-sum of $F$}.  Note that $\SS F$ depends only on $F$ but not on $S$.
		
		Finally, for the purposes of this paper, $F$ is called \textbf{summable} if there exists $k \ge 1$ such that $F$ is $k$-summable.
	\end{definition}
	
	\begin{remarks}
		\label{comp_rmk}
		Assume that $k \ge 1$ and $F$ is $k$-summable with its singular directions in $S$ and adopt the corresponding notations above.
		\begin{enumerate}
			\item It follows from Fact \ref{simple_summability}(2) and basic complex analysis that $F$ converges if and only if $F$ is summable and has no singular directions.  In this situation, we identify $\SS F$ with the germ at 0 of the analytic function defined by $F$.
			\item Let $G \in\Ps{C}{X}$ be \textit{convergent} and of order
			$\nu>0$.  Using Corollary \ref{ramis-sibuya}, we obtain (we leave the details to the reader) that the series $$(F \circ G)(X):= F(G(X))$$ belongs to $\Ps{C}{X}_{1/\nu k}$.  Moreover, the singular directions of $F \circ G$ belong to $S':= \bigcup_{\mu=0}^{\nu-1} (S + 2\pi\mu)/\nu$, and the corresponding sums and Stokes phenomena, for $\theta \in \S^1$ and $\ast \in \{+,-\}$, are $$\SS_{\theta,\theta^\ast(S')} (F \circ G) = \SS_{\nu\theta, (\nu\theta)^\ast(S)} F \circ \SS G$$ and $$\Delta_{\theta}(F \circ G) = \Delta_{\nu\theta} F \circ \SS G.$$
		\end{enumerate}
	\end{remarks}
	
	The next computation (Example \ref{explicit_stokes} below) is a crucial ingredient in our proof of Theorem \ref{th:main-SAT}.  Here and in Section \ref{sat}, we shall use the following:
	
	\begin{remark}
		\label{binom_rmk}
		Let $X = (X_1, \dots, X_m)$, $Y=(Y_1, \dots, Y_n)$ and $Z = (Z_1, \dots, Z_n)$, and let $F \in \Ps{R}{X,Y}$.  Then there are $B_1, \dots, B_n \in \Ps{R}{X,Y,Z}$ such that 
		\begin{equation*}
		F(X,Y) - F(X,Z) = \sum_{i=1}^n B_i(X,Y,Z) (Y_i-Z_i);
		\end{equation*}
		moreover, we have 
		\begin{equation*}
		B_i(X,Y,Y) = \frac{\partial F}{\partial Y_i}(X,Y).
		\end{equation*}
		The case $n=1$ follows from the binomial formula; for $n > 1$, proceed by induction on $n$ (simultaneously for all $m$), using the equality
		\begin{multline*}
		F(X,Y) - F(X,Z) = F(X,Y',Y_n) - F(X,Y',Z_n) \\ + F(X,Y',Z_n) - F(X,Z',Z_n),
		\end{multline*}
		where $Y':= (Y_1, \dots, Y_{n-1})$ and $Z':= (Z_1, \dots, Z_{n-1})$.  It follows, moreover, that the $B_i$ are convergent whenever $F$ is.
	\end{remark}
	
	\begin{example}[Stokes phenomena for $H$]
		\label{explicit_stokes}
		Assume that system \eqref{eq:system-gen} is in interlaced final form \eqref{eq:int-final-form}, with $r>0$, and let $H$ be its unique formal solution at 0 satisfying $H(0) = 0$.  As before, we set $$A(x):= a(x)I + x^r b(x) J,$$ and we also write $g(x,y) = \sum_{i=0}^\infty g_i(x) y^i$.  Following \cite{Ram-Sib}, each component of $H$ is $q$-summable with singular directions among the directions of the $q$th roots of the eigenvalues of $A(0)$. Since $r>0$, we have $a(0) \ne 0$ since $A(0)=a(0)I \ne 0$; hence, by assumption, $a(0)>0$. Therefore, the possible singular directions are the $q$th roots of unity $$S:= \left\{\frac{2p\pi}{q}:\ p=0,\dots,q-1\right\}.$$
		We refer to Definition \ref{stokes_def} for the corresponding sums $$\SS_{\theta,\theta^\ast(S)} H = (\SS_{\theta,\theta^\ast(S)} H_1,\SS_{\theta,\theta^\ast(S)} H_2)$$ and Stokes phenomena $$\Delta_\theta H = (\Delta_\theta H_1,\Delta_\theta H_2),$$ for $\theta \in \S^1$ and $\ast \in \{+,-\}$.
		Below, we set $Y:= (Y_1,Y_2)$ and $Z:= (Z_1,Z_2)$.  By Remark \ref{binom_rmk}, there is a convergent $B \in \mathcal{M}_2(\mathbb{R}[[X,Y,Z]])$ such that $$B(X,Y,Z)(Y-Z) =T\Theta(X,Y)-T\Theta(X,Z).$$ 
		
		Again by \cite{Ram-Sib}, $\SS_{\theta,\theta^\ast(S)} H$ is a solution of system \eqref{eq:system-gen} on $U(\theta,\theta^\ast,q)$; so $\Delta_\theta H$ is a solution of the system 
		\begin{equation}
		\label{eq:stokes_system}
		x^{q+1}y'=f_\theta(x) \cdot y
		\end{equation}
		on $V(\theta,q)$, where $f_\theta(x) := \SS B\left(x,\SS_{\theta,\theta^+(S)} H(x),\SS_{\theta,\theta^-(S)} H(x)\right)$ is a matrix with entries in $\mathcal{A}_{1/q}(V(\theta,q))$.   Thus, for 
		\begin{equation*}
		Q_a(x):= -\frac{1}{x^q}\left(\frac{a_0}{q}+\frac{a_1}{q-1}\,
		x+\dots+ \frac{a_{q-2}}2\,x^{q-2} + a_{q-1}\,x^{q-1}\right)
		\end{equation*}
		and $v \in \AA(V(\theta,q))^2$, we have that $w:= \exp(Q_a(x)) \cdot x^{a_q} \cdot v$ is a solution of system \eqref{eq:stokes_system} on $V(\theta,q)$ if and only if $v$ is a solution on $V(\theta,q)$ of the system
		\begin{equation}
		\label{eq:stokes_system2}
		x^{q-r+1}y'=\frac{f_\theta(x) - a(x)}{x^r} \cdot y.
		\end{equation} 
		Note that $T((f_\theta-a)/x^r)(X) = b(X) J + X^{q-r+1} L(X)$, where $L \in \mathcal{M}_2(\R[[X]])$; in particular, the linear part of $(f_\theta-a)/x^r$ has two distinct eigenvalues.  It follows from \cite[Theorem 12.2]{Was} that the system \eqref{eq:stokes_system2} can be diagonalized on $V(\theta,q-r) \supseteq V(\theta,q)$: there exists a holomorphic linear change of
		variables $v=C_\theta(x)u$, where $C_\theta \in \mathcal{M}_2(\AA(V(\theta,q)))$, such that $v \in \AA(V(\theta,q))^2$ satisfies \eqref{eq:stokes_system2} if and only if
		\begin{equation}
		\label{eq:stokes_system3}
		x^{q-r+1}u'=N_\theta(x)u,
		\end{equation}
		where $N_\theta \in \mathcal{M}_2(\AA(V(\theta,q)))$ is diagonal.  Moreover, from the Taylor expansion of $(f_\theta-a)/x^r$, we see that
		\begin{equation*}
		C_\theta(x) = \begin{pmatrix}1&1\\-i&i\end{pmatrix}+O(x^{q-r+1})
		\end{equation*}
		and
		\begin{equation*}
		N_\theta(x)=b(x) \begin{pmatrix}i & 0\\ 0 & -i\end{pmatrix} +O(x^{q-r+1}).
		\end{equation*}
		Setting $$Q_b(x) := \begin{cases} -\frac1{x^{q-r}} \left(\frac{b_0}{q-r} + \frac{b_1}{q-r-1}x + \cdots + b_{q-r-1} x^{q-r-1} \right) &\text{if } r < q, \\ 0 &\text{if } r=q, \end{cases}$$ the nonzero solutions of \eqref{eq:stokes_system3} are of the form $u = \mu_\theta \cdot E$, where $\mu_\theta = \diag(\mu_{\theta,1}, \mu_{\theta,2})$ with $\mu_{\theta,i} \in \AA(V(\theta,q)) \setminus \AA^0(V(\theta,q))$ and $E = \begin{pmatrix} e_{1} \\ e_{2} \end{pmatrix}$ with
		\begin{equation*}
		e_{1}(x) := \exp\big(i Q_b(x)\big) \cdot x^{ib_{q-r}} \quad\text{and}\quad e_{2}(x):= 1/e_{1}(x),
		\end{equation*}
		defined using the main branch of $\log$ on the sector $$\{z \in \C:\ |z|>0,\ \arg z \in V(\theta,q)\}.$$
		With these notations in place, we have shown that the Stokes phenomenon $\Delta_\theta H$ on $V(\theta,q)$, for singular $\theta \in \S^1$, is of the form
		\begin{equation}
		\label{eq:expression-stokes}
		\Delta_\theta H = (\exp \circ Q_a) \cdot x^{a_q} \cdot C_\theta \cdot \mu_\theta \cdot E,
		\end{equation}
		with $Q_a$, $a_q$ and $E$ depending only on the system \eqref{eq:system-gen} in interlaced final form \eqref{eq:int-final-form}, but not on the particular $\theta \in \S^1$.
	\end{example}
	
	\section{Multisummability}
	\label{multisummability}
	What happens if series of various summability orders are added or multiplied?  In general, the resulting series are not $k$-summable for any $k$; what happens instead is based on the ``relative Watson Lemma'':
	
	\begin{fact}[Proposition 2.1 of \cite{Mal-Ram}]
		\label{relative_Watson}
		Let $1/2 < k_1 < k_2$, and let $I \subseteq \S^1$ be an interval containing a closed interval of length $\pi/k_1$.  Then $\left(\AA^{k_1}/\AA^{k_2}\right)(I) = \{0\}$.
	\end{fact}
	
	To define multisummability, we use the following notation: let $J \subseteq I \subseteq \S^1$ be open intervals and $\BB \subseteq \CC$ be two sheaves on $\S^1$.  For $g \in \CC(I)$, we denote by $g\rest J$ the restriction of $g$ to $J$, and by $[g]_\BB$ the element of $(\CC/\BB)(I)$ represented by $g$.  Moreover, if $\DD$ is a third sheaf on $\S^1$ such that $\BB \subseteq \DD \subseteq \CC$, we identify $(\CC/\BB)/(\DD/\BB)$ with $\CC/\DD$ in the usual way (see \cite{Hartshorne:1977hr}).
	
	Let $1 \le k_1 <\dots <k_\mu$ and $F \in\Ps{C}{X}$, and set $k:= (k_1, \dots, k_\mu)$.  Recall \cite[D\'efinition 2.2]{Mal-Ram} that, if $I_1\supset I_2\supset\cdots 
	I_\mu$ are closed intervals on $\S^1$ such that each $I_\lambda$ has length at least $\pi/k_\lambda$ and $I:= (I_1, \dots, I_\mu)$, then
	$F$ is \textbf{$k$-summable on $I$} if $F \in \Ps{C}{X}_{1/k_1}$ and there exist $f_\lambda\in \left(\AA/\AA^{k_{\lambda+1}}\right)(I_\lambda)$, for
	$\lambda=1,\ldots,\mu-1$, and $f_\mu\in \AA(I_\mu)$ such that, if $f_0$ is the unique (see Corollary \ref{ramis-sibuya}) element of $\left(\AA/\AA^{k_1}\right)(\S^1)$ with $Tf_0 = F$, we have
	\begin{equation*}
	f_{\lambda-1}\rest{I_\lambda}\quad = \begin{cases} [f_\lambda]_{\AA^{k_\lambda}/\AA^{k_{\lambda+1}}} &\text{if } 1 \le \lambda < \mu, \\ [f_\lambda]_{\AA^{k_\lambda}} &\text{if } \lambda=\mu. \end{cases}
	\end{equation*}
	In this situation, it follows from Fact \ref{relative_Watson} that the tuple $f:= (f_1,\dots,f_\mu)$ is uniquely
	determined (\textbf{quasianalyticity}).  Thus, we call $f$ the \textbf{$k$-sum of $F$ on $I$} and, in particular, we set $\SS_{I} F := f_\mu \in \AA(I_\mu)$.

	\begin{definition}
		\label{def:multisummable}
		\begin{enumerate}
			\item Let $\theta \in \S^1$ and set $$I(\theta,k):= (I(\theta,k_1), \dots, I(\theta,k_\mu)).$$  Then $F$ is \textbf{$k$-summable} in the direction $\theta$ if $F$ is \hbox{$k$-summable} on $I(\theta,k)$.
			\item The series $F$ is \textbf{$k$-summable} if it is $k$-summable in all but finitely many directions; in this situation, the directions in which $F$ is not $k$-summable are called the \textbf{singular} directions of $F$.
			\item If $F$ is $k$-summable and $\xi,\zeta \in \S^1$ are such that $d(\xi,\zeta) < \pi$, and if the interval $U(\xi,\zeta)$ contains no singular directions of $F$ then, by analytic extension, there exists a unique $f \in \AA(U(\xi,\zeta,k_\mu))$ such that $f\rest{I(\theta,k_\mu)} = \SS_{I(\theta,k)} F$, for $\theta \in U(\xi,\zeta)$.  We call this $f$ the \textbf{$k$-sum of $F$ on $U(\xi,\zeta)$} and denote it by $\SS_{\xi,\zeta} F$.
		\end{enumerate}\medskip
		
		Let $S \subseteq \S^1$ be finite, and ssume that $F$ is $k$-summable with its singular directions in $S$.  As in the case of simple summability, for $\theta \in \S^1$ and $\ast \in \{+,-\}$, we define the \textbf{Stokes phenomenon of $F$ in the direction $\theta$} as $$\Delta_\theta F := \SS_{\theta,\theta^+(S)} F - \SS_{\theta,\theta^-(S)} F.$$ Note again that $\Delta_\theta F$ is independent of $S$, and that $\Delta_\theta F = 0$ whenever $\theta \notin S$. 
		
		By quasianalyticity, the tuples $\left(\SS_\theta^+ F\right)_{\theta \in \S^1}$ and $\left(\SS_\theta^- F\right)_{\theta \in \S^1}$ are uniquely determined by $F$ and $S$.
		Moreover, by definition, each $\Delta_\theta F$ belongs to $\AA^{k_1}(V(\theta,k_\mu))$.  It follows that the tuple $\left(\SS_\theta^+ F\right)_{\theta \in \S^1}$ represents and element in $\left(\AA/\AA^{k_1}\right)\left(\S^1\right)$, which we denote by $\SS F$ and call the \textbf{$k$-sum of $F$}.  Note that $\SS F$ depends on $F$ but not on $S$.
		
		Finally, for the purposes of this paper\footnote{If one replaces $\S^1$ by its universal covering space $\R$, all definitions and facts stated in this section are easily adapted to all tuples $k$ satisfying $k_1>0$; see \cite[Section 2]{Mal-Ram}.}, $F$ is \textbf{multisummable} if there exists a tuple $k$ as above such that $F$ is $k$-summable.
	\end{definition}
	
	\begin{remark}
		\label{multi_comp_rmk}
		It follows from quasianalyticity and basic complex analysis that $F$ converges if and only if $F$ is multisummable and has no singular directions.  
	\end{remark}  
	
	The collection of all multisummable series (as defined here) forms a subalgebra of $\Ps{C}{X}$ containing all summable series \cite[Section 2]{Mal-Ram}.  Moreover, by \cite[Proposition 2.3]{Mal-Ram}, this algebra is stable under composition on the left with convergent power series.  In a particular situation, as described next, we need a more precise statement of this kind.
	
	\subsection*{Composition of convergent with multisummable series}
	Let $m,n \in \N$ and $F \in \Ps{C}{X, X_{11}, \dots, X_{1n}, X_{21}, \dots, X_{mn}}$ be convergent; we abbreviate $$F(X,\{X_{ij}\}):= F(X, X_{11}, \dots, X_{1n}, X_{21},\dots, X_{mn}),$$ where
	$i=1,\ldots,m$ and $j=1,\ldots,n$.  Given $H_{ij} \in \Ps{C}{X}$ with $H_{ij}(0) = 0$, for each pair $(i,j)$, we set
	$$
	(F \circ \{H_{ij}\})(X):= F(X,\{H_{ij} (X)\}) \in \Ps{C}{X}.
	$$
	In this situation, if $J \subseteq \S^1$ is an interval and $h_{ij} \in \AA(J)$ are such that each $h_{ij}(0) = 0$, we write $\SS F \circ \{h_{ij}\}$ for the element $f \in \AA(J)$ represented by the function $z \mapsto \SS F(z,\{h_{ij}(z)\}):V \into \C$, for some appropriate $V \in \VV_J$. 
	
	Similarly, we need to define composition of $\SS F$ with elements of $\left(\AA/\AA^l\right)(\S^1)$: for $l \ge 1/2$, open intervals $J,J' \subseteq \S^1$, $\theta \in J \cap J'$ and $\alpha \in \AA(J)$ and $\beta \in \AA(J')$, note that $$\left([\beta]_{\AA^l}\right)_\theta = \left([\alpha]_{\AA^l}\right)_\theta \text{ if and only if } (\beta - \alpha)_\theta \in \left(\AA^l\right)_\theta.$$
	Thus, given $l > 1/2$ and $g_{ij} \in \left(\AA/\AA^{l}\right)(\S^1)$, for $1 \le i \le m$ and $1 \le j \le n$, 
	we define the \textbf{composition} $\SS F \circ \{g_{ij}\} \in \left(\AA/\AA^{l}\right)(\S^1)$ by setting, for $\theta \in \S^1$, 
	\begin{equation*}
	\left(\SS F \circ \{g_{ij}\}\right)_\theta := \left[\SS F \circ \{\alpha_{ij}\}\right]_{\AA^l}, 
	\end{equation*}
	where each $\alpha_{ij} \in \AA_\theta$ represents $(g_{ij})_\theta$.  This composition is well defined: if $\beta_{ij} \in \AA_\theta$ also represents $(g_{ij})_\theta$, then the polynomial growth of $\SS F$ implies that $$\left(\SS F \circ \{\beta_{ij}\} - \SS F \circ \{\alpha_{ij}\}\right)_\theta \in \left(\AA^l\right)_\theta,$$ as required.
	
	For the next proposition, we let $l \ge 1$ and $H_1,\ldots,H_m \in \Ps{C}{X}$ be $l$-summable, with corresponding sets $S_i \subseteq \S^1$ of singular directions and satisfying $H_i(0) = 0$.  Let also $P_1,\ldots, P_n \in\C[X]$ be polynomials satisfying $P_j(0)=0$.  We set $\nu_j:= \ord(P_j) > 0$ and denote by $1 \le k_1<k_2<\cdots<k_\mu$ the elements of the set $\{\nu_j l:\ j=1,\ldots,n\}$.
	
	By Remark \ref{comp_rmk}(2), each $H_i \circ P_j$ is $k(i,j)$-summable for some unique $k(i,j) \in \{k_1, \dots, k_\mu\}$ and $\SS(H_i \circ P_j) \in \left(\AA/\AA^{k(i,j)}\right)(\S^1)$, so that $$\left[\SS(H_i \circ P_j)\right]_{\AA^{k_1}/\AA^{k(i,j)}} \in \left(\AA/\AA^{k_1}\right)(\S^1).$$  We set $k:= (k_1, \dots, k_\mu)$ and let $S' \subseteq \S^1$ be the union of all directions associated to each $H_i \circ P_j$ as in Remark \ref{comp_rmk}(2) from the set $S_i$.  Note that, for $\theta \in \S^1$ and $i = 1, \dots, m$, we have $\theta^+(S_i) \ge \theta^+(S')$ and $\theta^-(S_i) \le \theta^-(S')$; setting $$\theta^\ast := \theta^\ast(S')$$ below, it follows that $U(\theta,\theta^\ast,k_\mu) \subseteq U(\theta,\theta^\ast(S_i),k(i,j))$ and the restriction $$\SS_{\theta,\theta^\ast}(H_i \circ P_j) \rest{U(\theta,\theta^*,k_\mu)} \in \AA(U(\theta,\theta^\ast,k_\mu))$$ is well defined.
	
	\begin{proposition}
		\label{pro:multisum}
		For $\theta \in \S^1$ and $\ast \in \{+,-\}$, the series $F \circ \{H_i \circ P_j\}$ is $k$-summable in every direction contained in $U(\theta,\theta^\ast)$ and satisfies $$\SS_{\theta,\theta^\ast} (F \circ \{H_i \circ P_j\}) = \SS F \circ \left\{ \SS_{\theta,\theta^\ast} (H_i \circ P_j) \rest{U(\theta,\theta^\ast,k_\mu)} \right\};$$
		in particular, the series $F \circ \{H_i \circ P_j\}$ is $k$-summable with singular directions among those in $S'$.
	\end{proposition}
	
	\begin{proof}
		We fix $\theta$, $\ast$ and $\phi \in U(\theta,\theta^\ast)$.  For $i\le m$, $j\le n$ and $\lambda \le \mu$, we define a sum $h^\lambda_{ij}$ of $H_i\circ P_j$ on the interval $I(\phi,k_\lambda) \subseteq U(\theta,\theta^\ast,k_\lambda)$ such that $h^\mu_{ij} \in \AA(I(\phi,k_\mu))$ and $h^\lambda_{ij} \in \left(\AA/\AA^{k_{\lambda+1}}\right)(I(\phi,k_\lambda))$ for $\lambda < \mu$, as follows:
		\begin{equation*}
		h^\lambda_{ij} := \begin{cases}
		\SS_{\theta,\theta^\ast} (H_i \circ P_j) \rest{I(\phi,k_\mu)} &\text{if } \lambda = \mu, \\
		\left[\SS_{\theta,\theta^\ast} (H_i \circ P_j)\right]_{\AA^{k_{\lambda+1}}}\rest{I(\phi,k_\lambda)} &\text{if } \lambda < \mu \text{ and } k(i,j) \le k_{\lambda+1}, \\
		[\SS(H_i \circ P_j)]_{\AA^{k_{\lambda+1}}/\AA^{k(i,j)}}\rest{I(\phi,k_\lambda)} &\text{if } k(i,j)> k_{\lambda+1}.
		\end{cases}
		\end{equation*}
		Then, for $1 \le \lambda \le \mu$, we set $f_\lambda:= \SS F \circ \left\{h^\lambda_{ij}\right\}$; in particular, $$f_\mu = \SS F \circ \left\{\SS_{\theta,\theta^\ast} (H_i \circ P_j) \rest{I(\phi,k_\mu)}\right\}$$ by definition.  It is straightforward from this definition that $F \circ \{H_i \circ P_j\}$ is $k$-summable in the direction $\phi$ with $k$-sum $(f_1, \dots, f_\mu)$ on $I(\phi,k)$.  The proposition now follows, since $\SS_{I(\phi,k)}(F \circ \{H_i \circ P_j\}) = f_\mu$ in this case.
	\end{proof}

	\section{Putting it all together} \label{sat}
	
	This section is devoted to the proof of Theorem \ref{th:main-SAT}; so we assume that system \eqref{eq:system-gen} is in interlaced final form \eqref{eq:int-final-form}, and we let $H$ be its unique divergent formal solution at 0 satisfying $H(0) = 0$.  By Remark \ref{0-sat_reduction}, it suffices to show that $H$ is 0-SAT.  As justified after the statement of Theorem \ref{th:main-SAT}, we shall assume throughout this proof that $r>0$.  We adopt all the notations introduced in Example \ref{explicit_stokes}.
	
	We now let $n \in \N$, $F \in \Ps{R}{X,Z}$ be nonzero and convergent, with $$Z = (Z_{ij})_{1 \le i \le 2, 1 \le j \le n},$$ and let $P_1,\dots,$ $P_n \in \R[X]$ be positive and $q$-short of orders $\nu_1, \dots, \nu_n>0$, respectively, and we adopt all corresponding notations introduced for Proposition \ref{pro:multisum} (with $m=2$, $l$ there equal to $q$ here and $S_i$ there equal to $S$ here) and Definition \ref{def:multisummable}.  We assume, for a contradiction, that 
	\begin{equation}
	\label{contradict_assumption}
	F \circ \{H_i\circ P_j\} =0.
	\end{equation}
	In this situation, we chose $F$ as follows: we let $\Lambda_F\subset\{1,2\}\times\{1,\ldots,n\}$ be the set of all
	indices $(i,j)$ such that $F$ depends on $Z_{ij}$, that is, the series obtained from $F$ by replacing the indeterminate $Z_{ij}$ with 0 is different from $F$.  Replacing $F$ if necessary, we may assume the cardinal $|\Lambda_F|$ is minimal among all non-zero convergent $F$ satisfying \eqref{contradict_assumption}, and we let $\FF$ be the set of all nonzero convergent power series $G(X,Z)$ such that $G \circ \{H_i \circ P_j\} = 0$ and $|\Lambda_G| = |\Lambda_F|$.
	
	The following lemma also appears on p. 437 of the proof of \cite[Theorem 4.4]{Rol-San-Sch}; we include its proof here for completeness' sake.
	
	\begin{lemma}\label{lm:partial-f-non-zero}
		Let $(i_0,j_0) \in \Lambda_F$.  There exists $G \in \FF$ such that
		\begin{equation*}
		\left(\partial G/\partial Z_{i_0j_0}\right) \circ \{H_i \circ P_j\} \ne 0.
		\end{equation*}
	\end{lemma}
	\begin{proof}
		Let  $(i_0,j_0) \in \Lambda_F$, and let $H_{i_0j_0}$ be the tuple obtained from the tuple $\{H_i \circ P_j\}$ after replacing the entry $H_{i_0} \circ P_{j_0}$ by the indeterminate $Z_{i_0j_0}$; in particular, $H_{i_0j_0} \in \Ps{R}{X,Z_{i_0j_0}}^{2n}$.  We claim that there exists $d\geq 1$ such that $\left(\partial^d F/\partial Z^d_{i_0j_0}\right) \circ \{H_i \circ P_j\} \ne 0$: otherwise, by Taylor expansion, the power
		series
		\begin{align*}
		F \circ H_{i_0j_0} &=
		\sum_{m\geq 0}\frac{1}{m!}\frac{\partial^m F}{\partial
			Z_{i_0j_0}^m} \circ \{H_i \circ P_j\} \cdot ((H_{i_0} \circ P_{j_0})-Z_{i_0j_0})^m \\ & = \sum_{k \ge 0} G_k(X) \cdot Z_{i_0j_0}^k
		\end{align*}
		in $\Ps{R}{X,Z_{i_0j_0}}$ is identically zero; in particular, each
		$G_k \in \Ps{R}{X}$ is zero. On the other hand, writing $Z' := \{Z_{ij}:\ (i,j) \ne (i_0,j_0)\}$, $H_{i_0j_0}' := \{H_i \circ P_j:\ (i,j) \ne (i_0,j_0)\}$ and $$F(X,Z) = \sum_{k \ge 0} F_k(X,Z') \cdot Z_{i_0j_0}^k,$$
		we see that $0 = G_k = F_k \circ H_{i_0j_0}'$ for each $k$.
		Since each $F_k$ converges and $|\Lambda_{F_k}| < |\Lambda_F|$, the minimality of $|\Lambda_F|$ implies that $F_k = 0$ for every $k$, hence $F = 0$, a contradiction. 
		
		Finally, we chose $d$ minimal such that $\left(\partial^d F/\partial Z^d_{i_0j_0}\right) \circ \{H_i \circ P_j\} \ne 0$, and we take $G:= \partial^{d-1}F/\partial
		Z_{i_0j_0}^{d-1}$.
	\end{proof}

	The rest of the proof is based on the following observation: recall that, for $\theta \in \S^1$, $\ast \in \{+,-\}$ and each $(i,j)$, the series $H_i \circ P_j$ is $k(i,j)$-summable in every direction contained in $U(\theta,\theta^\ast)$ and that $$V(\theta,k_\mu) \subseteq V(\theta,k(i,j)) = U(\theta,\theta^+,k(i,j)) \cap U(\theta,\theta^-,k(i,j)).$$
	\medskip
	
	\noindent\textbf{Claim.}
	\textit{There exist $\theta \in \S^1$ and $G \in \FF$ such that
		\begin{equation*}
		\SS G \circ \{\SS_{\theta,\theta^+}(H_i \circ P_j) \rest{V(\theta,k_\mu)}\} - \SS G \circ \{\SS_{\theta,\theta^-} (H_i\circ P_j) \rest{V(\theta,k_\mu)}\} \ne 0.
		\end{equation*}}
	
	Assuming this claim holds, we finish the proof of Theorem \ref{th:main-SAT} as follows: by Proposition \ref{pro:multisum}, the series $G \circ \left\{H_i \circ P_j\right\}$ is multisummable and satisfies
	\begin{multline*}
	\Delta_\theta(G \circ \{H_i \circ P_j\}) = \\ \SS G \circ \{\SS_{\theta,\theta^+} (H_i \circ P_j) \rest{V(\theta,k_\mu)}\} - \SS G \circ \{\SS_{\theta,\theta^-} (H_i\circ P_j) \rest{V(\theta,k_\mu)}\},
	\end{multline*}
	for $\theta \in \S^1$.  Thus, the claim implies that $G \circ \left\{H_i \circ P_j\right\}$ has at least one singular direction, so by Remark \ref{multi_comp_rmk}, the series $G \circ \left\{H_i \circ P_j\right\}$ is divergent, which contradicts the assumption that it is zero; this then proves the theorem.    
	
	\begin{proof}[Proof of the claim] 
		Let $Y = \{Y_{ij}\}$, for $(i,j) \in \{1,2\} \times \{1, \dots, n\}$; by Remark \ref{binom_rmk}, there are convergent $F_{ij} \in \Ps{R}{X,Y,Z}$ such that $$F(X,Y) - F(X,Z) = \sum_{(i,j)} F_{ij}(X,Y,Z) \cdot (Y_{ij} - Z_{ij}).$$  Therefore, for $\theta \in \S^1$, we get from Remark \ref{comp_rmk}(2) that
		\begin{align}
		\begin{split}
		\label{eq:sigmak}
		\SS F &\circ \{\SS_{\theta,\theta^+} (H_i \circ P_j) \rest{V(\theta,k_\mu)}\} - \SS F \circ \{\SS_{\theta,\theta^-} (H_i\circ P_j) \rest{V(\theta,k_\mu)}\} \\ &=\sum_{(i,j)}D_{ij,\theta} \cdot
		\left(\SS_{\theta,\theta^+} (H_i \circ P_j) \rest{V(\theta,k_\mu)} -\ 
		\SS_{\theta,\theta^-} (H_i\circ P_j) \rest{V(\theta,k_\mu)}\right) \\ &= \sum_{(i,j)} D_{ij,\theta} \cdot \Delta_\theta (H_i \circ P_j) \rest{V(\theta,k_\mu)} \\ &= \sum_{(i,j)} D_{ij,\theta} \cdot (\Delta_{\nu_j\theta} H_i \circ P_j) \rest{V(\theta,k_\mu)},
		\end{split}
		\end{align}
		where $D_{ij,\theta} \in \AA_{1/k_1}(V(\theta,k_\mu))$ has asymptotic expansion (see Remark \ref{binom_rmk})
		\begin{equation}\label{eq:dijk-asymptotics}
		T D_{ij,\theta} = F_{ij}(X, \{H_i \circ P_j\}, \{H_i \circ P_j\}) = \frac{\partial F}{\partial Z_{ij}}(X,\{H_i\circ P_j\}),
		\end{equation}
		independent of $\theta$.  
		We now chose $\theta \in \S^1$ such that $\nu_j\theta$ is a singular direction of $H$ for at least one $j \in \{1, \dots, n\}$; since $H$ is divergent, such $\theta$ and $j$ exist by Remark \ref{comp_rmk}(1).  Setting $$\Omega:= \{j: \nu_j\theta \text{ is a singular direction of } H\},$$ we obtain from 
		(\ref{eq:expression-stokes}) that, in restriction to $V(\theta,k_\mu)$,
		\begin{multline}\label{eq:sigmak-bis}
		\sum_{(i,j)} D_{ij,\theta} \cdot (\Delta_{\nu_j\theta} H_i \circ P_j) \\ = \sum_{j \in \Omega} \exp(Q \circ P_j)) \cdot P_j^{a_q} \cdot \begin{pmatrix} D_{1j,\theta} & D_{2j,\theta} \end{pmatrix} \cdot (C_\theta \circ P_j) \cdot (\mu_\theta \circ P_j) \cdot (E \circ P_j),
		\end{multline}
		where we write $Q := Q_a$.  
		
		To finish the proof of the claim, it suffices to find $\phi \in V(\theta,k_\mu)$ such that, after replacing $F$ by a suitable $G \in \FF$, the restriction of any representative of \eqref{eq:sigmak-bis} to the ray $R_\phi:= \{re^{i\phi}:\ r>0\}$ is not zero.  From the fact that the $P_j$ are distinct $q$-short real polynomials, we obtain the following (compare with p. 441 of the proof of \cite[Theorem 4.4]{Rol-San-Sch}):  \medskip
		
		\noindent\textbf{Subclaim.} For distinct $j_1, j_2 \in \Omega$, the meromorphic function $Q \circ P_{j_1} - Q \circ P_{j_2}$ has nonzero principal part at 0. 
		
		\begin{proof}
			We write $i$ instead of $j_i$ (for readability) and $P_{i}(z) = c_{\nu_i} z^{\nu_i} + \cdots + c_{d_i} z^{d_i}$ such that $\nu_i(q+1) > d_i$ and $c_{\nu_i} > 0$, for $i = 1,2$; in particular, $\nu_i > 0$.  Then
			\begin{equation*}
			(Q \circ P_{i})(z) = - \frac {a_0}{c_{\nu_i}^q} z^{-\nu_i q} + \text{ higher order terms},
			\end{equation*}
			so the subclaim follows if $\nu_1 \ne \nu_2$, or if $\nu_1 = \nu_2$ and $c_{\nu_1} \ne c_{\nu_2}$.  So we assume from now on that $\nu:= \nu_1 = \nu_2$ and $c_{\nu_1} = c_{\nu_2}$; then there exist $c \ne 0$ and $\nu < \mu < \nu(q+1)$ such that $P:= P_{2}-P_{1}$ satisfies
			\begin{equation*}
			P(z) = c z^\mu + \text{ higher order terms.}
			\end{equation*}
			Therefore,
			\begin{equation*}
			P_{2} = P_{1} + P = P_{1} \left(1 + \frac{P}{P_1}\right) = \frac{P_1}{1+\tilde P}
			\end{equation*}
			with $\tilde P \in \Ps{C}{X}$ of order $\mu-\nu \in (0,\nu q)$.  Setting $\tilde Q(z):= -z^q Q(z) \in \R[z]$, we obtain
			\begin{align*}
			Q \circ P_2 &= -\left(\frac{1+\tilde P}{P_1}\right)^q \cdot \tilde Q \circ (P_1+P) \\ &= -\left(\frac{1+\tilde P}{P_1}\right)^q \cdot \left( \tilde Q \circ P_1 + \frac{\tilde Q^{(1)} \circ P_1}{1!} P + \cdots + \frac{\tilde Q^{(q)} \circ P_1}{q!} P^q \right) \\ &= -\left(\frac{1+\tilde P}{P_1}\right)^q \cdot \left( \tilde Q \circ P_1 + O(x^\mu) \right) \\ &= (Q \circ P_1 ) - \frac1{P_1^q} \cdot \left( q\tilde P \cdot (\tilde Q \circ P_1) + O(z^{\mu-\nu+1})  \right).
			\end{align*}
			Now note that the term $\left( q\tilde P \cdot (\tilde Q \circ P_1)\right)/P_1^q$ belongs to $\C(\!(X)\!)$ and has order $\mu-\nu-\nu q < 0$, which finishes the proof of the subclaim.
		\end{proof}
		
		By the subclaim, there exists $\phi \in V(\theta,k_\mu)$ such that the germ at 0 of the restriction $q_{j_1,j_2}$ of the real part of $Q \circ P_{j_1} - Q \circ P_{j_2}$ to $R_\phi$ satisfies $\lim_{z \to 0} |q_{j_1,j_2}(z)| = \infty$, for distinct $j_1,j_2 \in \Omega$.  Thus, there is a unique $j_0 \in \Omega$ such that $\lim_{z \to 0} q_{j,j_0}(z) = -\infty$ for all $j \in \Omega \setminus \{j_0\}$; in particular, the germ at 0 of the restriction of
		$\frac{\exp(Q(P_{j}(z))}{\exp(Q(P_{j_0}(z))}$ to $R_\phi$ is exponentially flat for each such $j$.  Therefore, dividing
		(\ref{eq:sigmak-bis}) by $\exp(Q \circ P_{j_0})$, we see that it now suffices to prove, after replacing $F$ by a suitable $G \in \FF$, that the germ at 0 of the factor 
		\begin{equation}
		\label{factor}
		\begin{pmatrix} D_{1j_0,\theta} & D_{2j_0,\theta} \end{pmatrix} \cdot (C_\theta \circ P_{j_0}) \cdot (\mu_\theta \circ P_{j_0}) \cdot (E \circ P_{j_0})
		\end{equation}
		is not zero (since, in this case, the germ at 0 of this restriction is of polynomial growth, as shown in Example \ref{explicit_stokes}).
		
		By Lemma~\ref{lm:partial-f-non-zero} and
		(\ref{eq:dijk-asymptotics}), after replacing $F$ by a suitable $G \in \FF$, there exists $m \in \N$ such that 
		$$T D_{ij_0,\theta} = \alpha_i X^{m}+ \text{ higher order terms, for } i=1,2,$$
		and $\alpha_1$ and $\alpha_2$ are real and not both 0.  Similarly, by Example \ref{explicit_stokes}, there exists $m' \in \N$ such that $$T \mu_{\theta,i} \circ P_{j_0} = \beta_i X^{m'} + \text{ higher order terms, for } i=1,2,$$ and $\beta_1,\beta_2 \in \C$ are such that $\beta_1 \beta_2 \ne 0$.  Taking into account the form of the
		matrix $C_\theta(0)$ in Example \ref{explicit_stokes}, the factor \eqref{factor} is therefore equal to $\begin{pmatrix} \delta_1 & \delta_2 \end{pmatrix} \cdot (E \circ P_{j_0})$, where 
		\begin{align}
		\label{delta_form}
		\begin{split}
		\delta_1 &=\beta_1(\alpha_1+i\alpha_2) x^{m+m'}+ \epsilon_1 \\ \delta_2 &= \beta_2(\alpha_1-i\alpha_2) x^{m+m'}+ \epsilon_2
		\end{split}
		\end{align}
		with $\epsilon_i \in \AA(V(\theta,k_\mu))$ such that $\epsilon_i = o(x^{m+m'})$ as $x \to 0$, for $i=1,2$.  Since $e_2 = 1/e_1$ and $b_0 \ne 0$ in the definition of $e_1$, we get (working in the stalk over $\theta$, say) that $\begin{pmatrix} \delta_1 & \delta _2 \end{pmatrix} \cdot (E \circ P_{j_0}) = 0$ if and only if $(e_1 \circ P_{j_0})^2 = - \delta_2/\delta_1$, which is impossible by \eqref{delta_form}.
	\end{proof}
	\bibliographystyle{alpha}

	%
	%

\end{document}